\documentclass[12pt,a4paper]{amsart}
\usepackage{amsfonts}
\usepackage{amsthm}
\usepackage{amsmath}
\usepackage{amscd}
\usepackage[latin2]{inputenc}
\usepackage{t1enc}
\usepackage[mathscr]{eucal}
\usepackage{indentfirst}
\usepackage{graphicx}
\usepackage{graphics}
\usepackage{pict2e}
\usepackage{xcolor}
\usepackage{hyperref}
\usepackage{dsfont}
\newcommand{\E}{\mathbb{E}}

\hypersetup{
    colorlinks=true,
    linkcolor=blue,
    urlcolor=blue,
    linktoc=all
}

\usepackage{epic}
\numberwithin{equation}{section}
\usepackage[margin=2.5cm]{geometry}
\usepackage{epstopdf} 
\usepackage{cancel}
\usepackage{amsmath}
\usepackage{mathrsfs}
\usepackage{amsfonts}
\usepackage{dsfont}
\usepackage{amssymb}
\usepackage{graphicx}
\usepackage{tabu}
\usepackage{blindtext}
\usepackage{amsthm}
\usepackage{tikz}
\usepackage{relsize}
\usepackage{amsthm}
\usepackage{yhmath}
\usepackage[normalem]{ulem} 

\theoremstyle{plain}
\newtheorem{thm}{Theorem}[section]
\newtheorem{lemma}[thm]{Lemma}

\theoremstyle{definition}

\newtheorem{definition}[thm]{Definition}
\newtheorem{remark}[thm]{Remark}

\theoremstyle{remark}

\newtheorem*{stat*}{Statement}
\newtheorem*{prop*}{Proposition}
\numberwithin{equation}{section}

\newcommand{\Z}{\mathbb{Z}}

\newcommand{\R}{\mathbb{R}}

\newcommand{\var}{\text{Var}}

\makeatletter

\def\chaptermark#1{}

\def\chapter{%
  \if@openright\cleardoublepage\else\clearpage\fi
  \thispagestyle{plain}\global\@topnum\z@
  \@afterindenttrue \secdef\@chapter\@schapter}

\def\@chapter[#1]#2{\refstepcounter{chapter}%
  \ifnum\c@secnumdepth<\z@ \let\@secnumber\@empty
  \else \let\@secnumber\thechapter \fi
  \typeout{\chaptername\space\@secnumber}%
  \def\@toclevel{0}%
  \ifx\chaptername\appendixname \@tocwriteb\tocappendix{chapter}{#2}%
  \else \@tocwriteb\tocchapter{chapter}{#2}\fi
  \chaptermark{#1}%
  \addtocontents{lof}{\protect\addvspace{10\p@}}%
  \addtocontents{lot}{\protect\addvspace{10\p@}}%
  \@makechapterhead{#2}\@afterheading}
\def\@schapter#1{\typeout{#1}%
  \let\@secnumber\@empty
  \def\@toclevel{0}%
  \ifx\chaptername\appendixname \@tocwriteb\tocappendix{chapter}{#1}%
  \else \@tocwriteb\tocchapter{chapter}{#1}\fi
  \chaptermark{#1}%
  \addtocontents{lof}{\protect\addvspace{10\p@}}%
  \addtocontents{lot}{\protect\addvspace{10\p@}}%
  \@makeschapterhead{#1}\@afterheading}
\newcommand\chaptername{Chapter}

\def\@makechapterhead#1{\global\topskip 7.5pc\relax
  \begingroup
  \fontsize{\@xivpt}{18}\bfseries\centering
    \ifnum\c@secnumdepth>\m@ne
      \leavevmode \hskip-\leftskip
      \rlap{\vbox to\z@{\vss
          \centerline{\normalsize\mdseries
              \uppercase\@xp{\chaptername}\enspace\thechapter}
          \vskip 3pc}}\hskip\leftskip\fi
     #1\par \endgroup
  \skip@34\p@ \advance\skip@-\normalbaselineskip
  \vskip\skip@ }
\def\@makeschapterhead#1{\global\topskip 7.5pc\relax
  \begingroup
  \fontsize{\@xivpt}{18}\bfseries\centering
  #1\par \endgroup
  \skip@34\p@ \advance\skip@-\normalbaselineskip
  \vskip\skip@ }
\def\appendix{\par
  \c@chapter\z@ \c@section\z@
  \let\chaptername\appendixname
  \def\thechapter{\@Alph\c@chapter}}

\newcounter{chapter}

\newif\if@openright

\renewcommand{\tocsection}[3]{%
  \indentlabel{\@ifnotempty{#2}{\bfseries\ignorespaces#1 #2\quad}}\bfseries#3} 
\renewcommand{\tocsubsection}[3]{%
  \indentlabel{\@ifnotempty{#2}{\ignorespaces#1 #2\quad}}#3}

\newcommand\@dotsep{4.5}
\def\@tocline#1#2#3#4#5#6#7{\relax
  \ifnum #1>\c@tocdepth 
  \else
    \par \addpenalty\@secpenalty\addvspace{#2}%
    \begingroup \hyphenpenalty\@M
    \@ifempty{#4}{%
      \@tempdima\csname r@tocindent\number#1\endcsname\relax
    }{%
      \@tempdima#4\relax
    }%
    \parindent\z@ \leftskip#3\relax \advance\leftskip\@tempdima\relax
    \rightskip\@pnumwidth plus1em \parfillskip-\@pnumwidth
    #5\leavevmode\hskip-\@tempdima{#6}\nobreak
    \leaders\hbox{$\m@th\mkern \@dotsep mu\hbox{.}\mkern \@dotsep mu$}\hfill
    \nobreak
    \hbox to\@pnumwidth{\@tocpagenum{\ifnum#1=1\bfseries\fi#7}}\par
    \nobreak
    \endgroup
  \fi}
\AtBeginDocument{%
\expandafter\renewcommand\csname r@tocindent0\endcsname{0pt}
}
\def\l@subsection{\@tocline{2}{0pt}{2.5pc}{5pc}{}}
\makeatother

\usepackage{lipsum}

\begin{document}

\title{Gaussian Fluctuation for Smoothed Local Correlations in CUE}

\author[A. Soshnikov]{Alexander Soshnikov}

\address{University of California at Davis \\ Department of Mathematics \\ 1 Shields Avenue \\  Davis CA 95616 \\ United States of America} 

\email{soshniko@math.ucdavis.edu}



\begin{abstract}

Motivated by the Rudnick-Sarnak theorem we study limiting distribution of 
smoothed local correlations of the form 
$$ \sum_{j_1, j_2, \ldots, j_n} f(N\*(\theta_{j_2}-\theta_{j_1}), N\*(\theta_{j_3}-\theta_{j_1}), \ldots, N\*(\theta_{j_n}-\theta_{j_1}))$$
for the Circular United Ensemble of 
random matrices for sufficiently smooth test functions.
\end{abstract}
\maketitle


\section{Introduction and Formulation of the Main Result.}

        The Circular $\beta$-Ensemble (C$\beta$E) 
is determined by the joint probability density
            \begin{align} \label{betaensemble}
                p_N^\beta(\theta_1, \ldots, \theta_N)=\frac{1}{Z_{N,\beta}}\prod_{1\leq j< k\leq N}\left|e^{i\theta_j}-e^{i\theta_k}\right|^\beta, \ \ \  \end{align}
where $\theta=(\theta_1, \ldots, \theta_N) \in  \mathbb{T}^N, $ i.e. $0\leq \theta_1, \ldots, \theta_N <2\*\pi, \ \beta>0, $ and  the partition function $Z_{N,\beta}$ is given by
            \begin{align}
\label{selberg}
                Z_{N,\beta}=\int_{\mathbb{T}^N} \prod_{1\leq j< k\leq N}\left|e^{i\theta_j}-e^{i\theta_k}\right|^\beta \*d\theta= 
(2\pi)^N \* \frac{\Gamma\left(1+\frac{\beta N}{2}\right)}{\Gamma\left(1+ \frac{\beta}{2}\right)^N}.
       \end{align}

The ensemble was introduced in Random Matrix Theory by Dyson in \cite{Dyson}. 
The Circular Unitary Ensemble (CUE) corresponds to $\beta=2.$
In this special case
(\ref{betaensemble}) describes the joint distribution of the eigenvalues of an $N\times N$ random unitary matrix U distributed according to the Haar measure.

In \cite{pairs}-\cite{AS2}, we studied the limiting distribution of pair counting functions
\begin{align}
\label{pairs}
                S_N(f)=\sum_{1\leq i\neq j\leq N} f(L_N\*(\theta_i-\theta_j)_c),
            \end{align}
in the Circular $\beta$-Ensemble, where $1\leq L_N\leq N$
and
$(\theta_i-\theta_j)_c$ is the phase difference on the unit circle, i.e.
\begin{align}
\label{circlediff}
(x-y)_c= \begin{cases}  x-y&\text{if }  -\pi\leq x-y<\pi,\\
 x-y -2\*\pi  &\text{if  }  \pi\leq x-y<2\*\pi,\\
x-y +2\*\pi&\text{if}  -2\*\pi<x-y<-\pi.
\end{cases}
\end{align}

The case $\beta=2, \ L_N=N$ was motivated by
a classical result of Montgomery on pair correlation of zeros of the Riemann zeta function \cite{montgomery1}-\cite{montgomery2}. Let us denote the ``non-trivial'' zeros as $\{ 1/2 \pm i\* \gamma_k, \ \gamma_k>0\}.$   Montgomery's theorem suggested that the rescaled zeros
\[ 
\tilde{\gamma}_k= \frac{\gamma_k}{2\*\pi} \*\log(\gamma_k),
\]
asymptotically behave as the eigenvalues of a large random unitary (CUE) matrix. 
Namely, Montgomery studied the asymptotic behavior of the the statistic
\[
F_T(\alpha)=T^{-1}\* \sum_{0<\tilde{\gamma}_j, \tilde{\gamma_k}\leq T} \exp(i\*\alpha\*(\tilde{\gamma}_j-\tilde{\gamma}_k)) 
\*\frac{4}{4+(\tilde{\gamma}_j-\tilde{\gamma}_k)^2/\log(T)^2},
\]
for real $\alpha$ and large real $T.$ 
Assuming the Riemann Hypothesis, Montgomery argued that
\begin{align}
\label{montform}
F(\alpha):=\lim_{T\to\infty} F_T(\alpha)=\begin{cases}  |\alpha|, &\text{ if }  \   0<|\alpha|<1    \\
                                                        1 , &\text{if} \ |\alpha|\geq 1.
                    \end{cases}. 
\end{align}
Montgomery rigorously proved (\ref{montform}) for $0<|\alpha|<1$ and provided heuristic arguments to support the formula for $|\alpha|\geq 1.$ He also proved that for 
$ \alpha=0 $ and large $T$ the statistic behaves as
\[ F_T(0)=\log(T)^{-1} \*(1+o(1)).\]

Since $\alpha \mapsto \delta(\alpha)+\min(|\alpha|, 1)$ is the Fourier transform of $\delta(x)+ 1-\left(\frac{\sin(\pi\*x)}{\pi\*x}\right)^2,$ the formula (\ref{montform}) implies that
the asymptotic behavior of the two-point correlations of the rescaled zeros of the Riemann zeta functions and the eigenvalues of a large random unitary matrix coincide in the limit.  

In 1994, Hejhal (\cite{hejhal}) extended the result to the case of three-point correlation functions, again under a technical condition on the support of the Fourier transform of a test function.
In 1996, Rudnick and Sarnak (\cite{RS}) considered the $l$-level correlation sum, $l\geq 2,$
\begin{align}
\label{rudsar1}
R_{T,l}(g)=\sum^{*} g(\tilde{\gamma}_{j_1}, \ldots, \tilde{\gamma}_{j_l}),
\end{align}
where the sum is restricted to distinct rescaled zeros in the interval $[0,T]$ for large $T,$
and $g$ is a smooth test function satisfying:

(i) $g(x_1, x_2, \ldots, x_l)$ is symmetric.

(ii) $g(x_1+t, x_2+t, \ldots, x_l+t)=g(x_1, x_2, \ldots, x_l)$ for all $t\in \R.$

(iii) $g(x_1, x_2, \ldots, x_l)\to 0$ rapidly as $|x|\to\infty$ in the hyperplane $\sum x_j=0.$

Assuming the Riemann Hypothesis and an additional important technical condition that $\hat{g}(\xi_1, \ldots, \xi_l)$, the Fourier transform of $g$, is supported in $\sum_j |\xi_j|<2,$ they proved that

\begin{align}
\label{rudsar2}    
 R_{T,l}(g) \to \int_{\R^l} g(x_1, \ldots, x_l)\*\rho_l(x_1, \ldots, x_l) \* \delta(\frac{x_1+\ldots+x_l}{l}) \* dx_1\cdots dx_l, 
\end{align}
where $\delta(x)$ is the delta function and 
\begin{align}
\label{rudsar3}   
\rho_l(x_1, \ldots, x_l)=det(K(x_i,x_j))_{1\leq i,j\leq l}, \ \ K(x,y)=\frac{sin(\pi\*(x-y))}{\pi(x-y)},
\end{align}
is the limiting $l$-point correlation function for the CUE. We refer the reader to \cite{conrey}, \cite{MS}, \cite{keating}, \cite{CNN} and the references therein for additional information on exciting connections between Random Matrices and Number Theory.

Introducing $f(y_1, \ldots, y_{n})=g(0, y_1, \ldots, y_{l-1}), \ n=l-1, \ $ we can rewrite (\ref{rudsar1}) as
\begin{align}
\label{rudsar4}
R_{T,l}(g)=\sum^{*} f(\tilde{\gamma}_{j_2}-\tilde{\gamma}_{j_1}, \ldots, \tilde{\gamma}_{j_l}-\tilde{\gamma}_{j_1}),
\end{align}
where $f$ is a smooth symmetric test function decaying fast at infinity and the sum is over all
distinct rescaled zeros in the interval $[0,T].$
In this paper, we study the limiting fluctuation of the analogue of (\ref{rudsar4}) for the CUE.
Specifically, we consider
\begin{align}
\label{kfoldstat}
S_N(f)=\sum_{1\leq j_1, j_2, \ldots, j_{n+1}\leq N} f(N\*(\theta_{j_2}-\theta_{j_1})_c, \ldots, N\*(\theta_{j_{n+1}}-\theta_{j_1})_c),
\end{align}
where $\theta=(\theta_1, \ldots, \theta_N) \in  \mathbb{T}^N$ comes from the CUE and $f\in C^{\infty}_c(\R^{n})$ is a smooth test function with compact support.
Even though the number of terms in (\ref{kfoldstat}) is $N^{n+1},$ the number of non-zero terms in the sum is of order $N.$

When $n=1$ the following result was proven in \cite{pairs}:

\begin{thm} (\cite{pairs}). 

Let $f \in C^{\infty}_c(\R)$ be an even, smooth, compactly supported function on the real line and $\theta=(\theta_1, \ldots, \theta_N) \in  \mathbb{T}^N$ be CUE-distributed. Consider
\begin{align}
\label{localcase}
S_N(f(N\cdot))=\sum_{1\leq i,j\leq N} f(N\*(\theta_i-\theta_j)_c).
\end{align}
Then 
\[\E S_N(f)= \sum_{k\in \mathbb{Z}} \frac{1}{\sqrt{2\pi}}\hat{f}(k/N) \*\min\left(\frac{|k|}{N}, 1\right) + \hat{f}(0)\*N, 
\]
and
$(S_N(f(N\cdot)) -\E S_N(f(N\cdot)))\*N^{-1/2}$ converges in distribution to the centered real Gaussian random variable with the variance
\begin{align}
\label{vvvv}
\frac{1}{\pi}\* \int_{\mathbb{R}} |\hat{f}(t)|^2 \*\min(|t|,1)^2\* dt -\frac{1}{\pi}\* \int_{|s-t|\leq 1, |s|\vee|t|\geq 1} \hat{f}(t)\*\hat{f}(s)\*
(1-|s-t|)\* ds\*dt\\
-\frac{1}{\pi}\* \int_{0\leq s,t\leq 1, s+t>1} \hat{f}(s)\*\hat{f}(t)\*(s+t-1) \*ds\*dt. \nonumber
\end{align}
\end{thm}
\begin{remark}
The limiting distribution of (\ref{localcase}) does not change if one replaces the circular difference (\ref{circlediff}) in the argument of $f(N\cdot)$
by the regular one and studies instead
\begin{align}
\label{localnoncirc}
\sum_{1\leq i,j\leq N} f(N\*(\theta_i-\theta_j)), 
\end{align}
since the number of pairs of the eigenvalues in a $O(N^{-1})$ neighborhood of $\theta=0$ is bounded in probability.
\end{remark}

The proof relied on the cumulant technique and had a strong combinatorial flavor. 
The purpose of this paper is to give a simpler proof that works for arbitrary $n>1.$
Below we formulate our main result.

\begin{thm}

Let $f \in C^{\infty}_c(\R^{n})$ be a smooth, symmetric, compactly supported function on $\R^{n}, \ 
n\geq 1, $ and $\theta=(\theta_1, \ldots, \theta_N) \in  \mathbb{T}^N$ be a CUE-distributed random vector.
Consider the $l$-tuple smoothed counting statistic $ S_N(f), \ l=n+1,$ defined in (\ref{kfoldstat}).  
Then  $\E S_N(f)$ satisfies (\ref{mathexp2}-\ref{c2}) and the normalized random variable $(S_N(f) -\E S_N(f))\*N^{-1/2}$ converges in distribution to the centered real Gaussian random variable $N(0, \sigma^2(f)$  with the limiting variance $\sigma^2(f)$ defined in (\ref{v4}).
\end{thm}

The paper is organized as follows.  Some preliminary facts are given in Section 2.  Section 3 is devoted to the computation of mathematical expectation  and variance of $S_N(f).$ Theorem 1.3 is proven in Section 4.  Throughout the paper the notation $a_N=O(b_N)$ means 
that the ratio $ a_N/b_N$ is bounded from above in absolute value. The notation $a_N=o(b_N)$ means that $a_n/b_N\to 0$ as $N\to \infty.$

This material is based upon work supported by the National Science Foundation under Grant No. 1440140, while the author was in residence at the Mathematical Sciences Research Institute in Berkeley, California, during a part of the Fall semester of 2021.

\section{Preliminary Facts.}

Let $f \in C^{\infty}_c(\R^{n}).$  When $N$ is sufficiently large, the support of $f(N\*x), \ x=(x_1, \ldots, x_{n} ), $ is contained in the cube $[-\pi,\pi]^{n}$ and
one can write a Fourier series
            \[
                f(N\*x)=f(N\*x_1, \ldots, N\*x_n)=\frac{1}{(2\pi)^{n/2}\*N^{n}}\*\sum_{k\in\mathbb{Z}^n} \* \hat{f}(k_1\*N^{-1}, \ldots, k_n\*N^{-1})e^{i\*k\cdot x},
            \]
        where 
            \[
                \hat{f}(\xi)=\frac{1}{(2\pi)^{n/2}}\*
                \int_{\mathbb{R}^n}f(x)e^{-i\*\xi\cdot x}\*dx.
            \]
        
Then
            \begin{align}
                S_N(f) &=&\sum_{1\leq j_1, j_2, \ldots, j_{n+1}\leq N} f(N\*(\theta_{j_2}-\theta_{j_1})_c, \ldots, N\*(\theta_{j_{n+1}}-\theta_{j_1})_c) \nonumber\\
                \label{f2}
                &=& \frac{1}{(2\pi)^{n/2}\*N^{n}}\*
\sum_{k\in \mathbb{Z}^n} \hat{f}(k_1\*N^{-1}, \ldots, k_n\*N^{-1})\*\prod_{j=1}^{n+1} T_{N, k_j},
            \end{align}
            where
            \begin{align}
            \label{deftrace}
                T_{N,s}=\sum_{m=1}^N e^{i\*s\*\theta_m}=Tr (U^s)
            \end{align}
is the trace of the $s$-th power of a random unitary (CUE) matrix, and $k_{n+1}=-\sum_{j=1}^n k_j.$

First, we evaluate $\E S_N(f).\ $ One has
\begin{align}
            \label{f3}
  \E S_N(f)=  \frac{1}{(2\pi)^{n/2}\*N^{n}}\*
\sum_{k\in \mathbb{Z}^n} \hat{f}(k_1\*N^{-1}, \ldots, k_n\*N^{-1})\*\times\E [\prod_{j=1}^{n+1} T_{N, k_j}].        
            \end{align}
To compute $\E [\prod_{j=1}^{n+1} T_{N, k_j}],$   we study the joint cumulants of the traces of powers of a CUE matrix. We refer the reader to \cite{malyshev} (see also \cite{pairs}, Section 5) for the definition and basic properties of the joint cumulants.   We will use the notation $\kappa_m^{(N)}(k_1, \ldots, k_m)$     for the joint cumulant of $T_{N, k_1}, \ldots, T_{N, k_m}, $ i.e.
\[
\kappa_m^{(N)}(k_1, \ldots, k_m)=\kappa(T_{N, k_1}, \ldots, T_{N, k_m}). 
\]
Recall that
\begin{align}
\label{cumtomom}
\E [\prod_{j=1}^{n+1} T_{N, k_j}]=\sum_{\pi} \prod_{B\in \pi} \kappa_{|B|}^{(N)}(k_i: i \in B),
\end{align}
where the sum is over all partitions $\pi$ of $\{1, \ldots, n+1\},$ $B$ runs through the list of all blocks of the partition $\pi, \ $ and $|B|$ is the cardinality of a block $B.$
The following result was established in \cite{sasha} (see also Section 5 of \cite{pairs}):

\begin{lemma}
(i)  If $p>1$ and  either $k_1+\ldots +k_p\neq 0$ or $ \ \prod_i^p k_i=0, \ $ or both, then
\begin{align}
\label{semi1} 
\kappa_p^{(N)}(k_1, \ldots, k_p)=0.
\end{align}

(ii) If $p>1, \ \ \prod_i^p k_i\neq 0,$ and $\sum_{i=1}^p k_i=0,$ then
\begin{align}
\label{semi2} 
\kappa_p^{(N)}(k_1, \ldots, k_p)=\sum_{m=1}^p \frac{(-1)^{m}}{m} \*\sum_{\substack{(p_1, \ldots, p_m):
\\ p_1+\ldots +p_m=p, \ p_1, \ldots p_m\geq 1}} 
\frac{1}{p_1!\cdots p_m!} \* 
\sum_{\sigma \in S_p} \*  J_N(p_1, \ldots, p_m; k_{\sigma(1)},\ldots, k_{\sigma(p)}),
\end{align}

where for positive integers   $p_1, \ldots, p_m\geq 1, \ p_1+\ldots+p_m=p ,$ and integers $k_1, \ldots, k_p,$ satisfying $\sum_{i=1}^p k_i=0,$ we define
\begin{align}
& J_N(p_1, \ldots, p_m; k_1,\ldots, k_p)
:= \nonumber\\
\label{J1}
& \min\left(N, \ \max\left(0, \sum_{i=1}^{p_1} k_i, \sum_{i=1}^{p_1+p_2} k_i, \ldots, \sum_{i=1}^{p_1+\ldots+p_{m-1}} k_i\right)+
\max\left(0, \sum_{i=1}^{p_1} (-k_i), \ldots, \sum_{i=1}^{p_1+\ldots+p_{m-1}} (-k_i)\right)\right).
\end{align}

(iii) If $p=1$ then $\kappa_1^{(N)}(k)=N$ for $k=0$ and $\kappa_1^{(N)}(k)=0$ otherwise.
\end{lemma}

\begin{remark}
If $\sum_{i=1}^p k_i=0$ and $\sum_{i=1}^p |k_i|\leq 2\*N,$ then
\begin{align}
& J_N(p_1, \ldots, p_m; k_1,\ldots, k_p)= \nonumber \\
\label{J2}
& \max\left(0, \sum_{i=1}^{p_1} k_i, \sum_{i=1}^{p_1+p_2} k_i, \ldots, \sum_{i=1}^{p_1+\ldots+p_{m-1}} k_i\right)+
\max\left(0, \sum_{i=1}^{p_1} (-k_i), \ldots, \sum_{i=1}^{p_1+\ldots+p_{m-1}} (-k_i)\right).
\end{align}

Using a combinatorial identity (Lemma 2 in \cite{sasha})  one further obtains that
\begin{align}
\label{J3}
\kappa_p^{(N)}(k_1, \ldots, k_p)=0 \ \ {\text if} \ \   p>2   \ \ {\text and} \ \ \sum_{i=1}^p |k_i|\leq 2\*N.  
\end{align}
\end{remark}
\begin{remark}
It directly follows from (\ref{J1})  that 
$$0\leq J_N(p_1, \ldots, p_m; k_1,\ldots, k_p)\leq N.$$
Therefore (\ref{semi1}-\ref{semi2}) imply
\begin{align}
\label{J4}
|\kappa_p^{(N)}(k_1, \ldots, k_p)|\leq Const_p \*N,
\end{align}
where $Const_p$ depends only on $p.$
\end{remark}

\begin{remark}
Recall that $\kappa_p^{(N)}(k_1, \ldots, k_p)$ is a symmetric function of $k_1, \ldots, k_p$.  Direct computations allow one to derive joint cumulants in several important special cases:
\begin{align}
\label{j1}
& \kappa_2^{(N)}(k_1, -k_1)=min(N, |k_1|), \\
\label{j2}
& \kappa_3^{(N)}(k_1, k_2, -(k_1+k_2))=\begin{cases} 0, &\text{ if } k_1+k_2\leq N, \hspace{2mm} k_1, k_2\geq 0,\\
                                  k_1+k_2-N, &\text{ if } k_1+k_2>N,\hspace{2mm} 0\leq k_1, k_2 \leq N,   \\                                  
                                  k_2, &\text{if  }   k_1+k_2>N,\hspace{2mm} k_1>N, \hspace{2mm} 0\leq k_2 \leq N,  \\
                                  N, &\text{  } K_1\geq N,\hspace{2mm} k_2\geq N.
                                 \end{cases}.\\
\label{j3}
& \kappa_4^{(N)}(k_1, k_2, -k_1, -k_2)=\begin{cases} 0, &\text{ if } 1\leq |k_1|=|k_2|\leq N/2,\\
                                  N -2|k_1|, &\text{ if } N/2< |k_1|=|k_2|\leq N,\\                                  -N , &\text{if  }  |k_1|=|k_2|\geq N,\\
                                  ||k_1|-|k_2||-N, &\text{ if } 1\leq ||k_1|-|k_2||\leq N-1,\hspace{2mm}N\leq \max(|k_1|,|k_2|),\\
                                  N-|k_1|-|k_2|, &\text{if  } 1\leq |k_1|\neq |k_2|\leq N-1, N+1\leq |k_1|+|k_2|,\\
                                  0,&\text{else.}
                                   \end{cases}.
                                 \end{align}
                                 
We note that (\ref{j3}) directly follows  from Corollary 4.2 in (\cite{pairs}). 
(\ref{j2}) follows 
from (\ref{semi2}) and (\ref{J2}). Since $\kappa_p^{(N)}(k_1, \ldots, k_p)$  is a symmetric function that satisfies
$\kappa_p^{(N)}(k_1, \ldots, k_p)=\kappa_p^{(N)}(-k_1, \ldots, -k_p),$  the third cumulant formula (\ref{j2}) completely determines $\kappa_3^{(N)}(k_1, k_2, k_3)$ (recall that $\kappa_3^{(N)}$ vanishes if $k_3\neq -k_1-k_2$.)  The second cumulant formula (\ref{j1}) immediately follows from  (\ref{semi2}) and (\ref{J2}).                           \end{remark}          
   
   \vskip .1cm         
The main  contribution to (\ref{f2}) and (\ref{f3}) comes from $|k|=O(N).$ Define
\begin{align}
\label{ttox}
t_i=\frac{k_i}{N}, \ \ 1\leq i\leq p.
\end{align}
Then
\begin{align}
\label{scalcum}
c_p(t_1, \ldots, t_p):= \frac{1}{N}\* \kappa_p^{(N)}(t_1\*N, \ldots, t_p\*N)
\end{align}
does not depend on $N$ and is a bounded  function of $t_1, \ldots, t_p, $ which is identically zero on $t_1+\ldots+t_p\neq 0$ for $p>1$ and is piece-wise linear on 
$t_1+\ldots+t_p=0.$
To write down an explicit formula for $c_p(t_1, \ldots, t_p)$ we define functions $j(p_1, \ldots, p_m; t_1,\ldots, t_p)$
for positive integers   $p_1, \ldots, p_m,$ satisfying $\ p_1+\ldots+p_m=p ,$ and real numbers $t_1, \ldots, t_p,$ as
\begin{align}
& j(p_1, \ldots, p_m; t_1,\ldots, t_p)
:= \nonumber\\
\label{jj1}
& \min\left(1, \ \max\left(0, \sum_{i=1}^{p_1} t_i, \sum_{i=1}^{p_1+p_2} t_i, \ldots, \sum_{i=1}^{p_1+\ldots+p_{m-1}} t_i\right)+
\max\left(0, \sum_{i=1}^{p_1} (-t_i), \ldots, 
\sum_{i=1}^{p_1+\ldots+p_{m-1}} (-t_i)\right)\right).
\end{align}
Lemma 2.1 implies the following result.

\begin{lemma}
Rescaled joint cumulants $c_p$   defined in (\ref{scalcum}) can be written for $p>1$ and $\sum_{i=1}^p t_i=0$ as
\begin{align}
\label{jj2}
c_p(t_1, \ldots, t_p)=\sum_{m=1}^p \frac{(-1)^{m}}{m} \*\sum_{\substack{(p_1, \ldots, p_m):
\\ p_1+\ldots +p_m=p, \ p_1, \ldots p_m\geq 1}} 
\frac{1}{p_1!\cdots p_m!} \* 
\sum_{\sigma \in S_p} \*  j(p_1, \ldots, p_m; t_{\sigma(1)},\ldots, t_{\sigma(p)}),
    \end{align}
    where the functions $j(p_1, \ldots, p_m; t_1,\ldots, t_p)$ are defined in (\ref{jj1}).
    Moreover, the following holds:
  \begin{align}  
   & (i) \ \ c_p(t_1, \ldots, t_p), \ p>1, \ \text{ is a bounded symmetric piece-wise linear function on} \ \ 
   \sum_{i=1}^p t_i=0. \nonumber\\
  & (ii) \ \ c_p(t_1, \ldots, t_p)=0 \ \text{if} \ \ p>1 \ \ \text{and} \ \ \sum_{i=1}^p t_i\neq 0.\nonumber\\
   & (iii) \ \ c_1(0)=1 \ \ \text{and} \ \ c_1(t)=0 \ \ \text{for} \ \ t\neq 0. \nonumber
   \end{align}
\end{lemma}

\section{Expectation and Variance.}

We start with a computation of $\E S_N(f).$ It follows from (\ref{f3}), (\ref{cumtomom}), (\ref{scalcum}), and (\ref{j2}) that
\begin{align}
\label{c1}
  \E S_N(f) & =   \frac{1}{(2\pi)^{n/2}\*N^{n}}\*
\sum_{k_1, \ldots, k_n\in \mathbb{Z}} \hat{f}(k_1/N, \ldots, k_n/N)\*
  \sum_{\pi} N^{|\pi|}\prod_{B\in \pi}\*  c_{|B|}(k_i/N: i \in B) \\
  & =  \frac{1}{(2\pi)^{n/2}}\*\sum_{\pi}\* N^{|\pi|-n}\* \sum_{k_1, \ldots, k_n\in \mathbb{Z}} \hat{f}(k_1/N, \ldots, k_n/N)\*\prod_{B\in \pi}\*  c_{|B|}(k_i/N: i \in B), \nonumber 
 \end{align}
where $k_{n+1}=-\sum_{i=1}^n k_i, \ $ the sum is over all partitions $\pi$ of $\{1, \ldots, n+1\},$ $B$ runs through the list of all blocks of the partition $\pi,\ |\pi| \ $ is the number of blocks in a partition $\pi,$ and $|B|$ is the cardinality of a block $B.$ 

Denote a linear subspace of $\R^{n+1}$ as
\begin{align}
\label{lb}
L_{\pi}:=\{t=(t_1, \ldots, t_{n+1})\in\R^{n+1} : \sum_{i=1}^{n+1} t_i=0; \ \ \sum_{i\in B} t_i=0 \ \forall B\in \pi \}.
\end{align}
It follows from Lemma 2.5, parts (ii)-(iii), that for any fixed partition $\pi$ the summation in (\ref{c1}) is over $k=(k_1, \ldots, k_{n+1}) \in L_{\pi} \cap \frac{1}{N}\*\Z^{n+1}.$ Indeed, $k$ satisfies the following system of linear equations of rank $|\pi|:$
\begin{align}
\label{lineq}
\begin{cases}
& \sum_{i=1}^{n+1} k_i=0, \\
& \sum_{i\in B} k_i=0, \ \ \forall B\in \pi,
\end{cases}
\end{align}
Observe that the first linear equation in (\ref{lineq}) follows from the remaining $|\pi|$ independent linear equations in (\ref{lineq}). Therefore
\begin{align}
\label{diml}
\dim L_{\pi}=n+1-|\pi|.
\end{align}

We note that the sums in (\ref{c1}) are Riemann sums (up to a multiplicative factor 
$\frac{N}{(2\pi)^{n/2}}$)
corresponding to smooth and fast decaying at infinity functions
$\hat{f}(t_1, \ldots, t_n)\*
\prod_{B\in \pi}\*c_{|B|}(t_i: i \in B)$ on $L_{\pi}$.
Denote by $m$ the number of blocks in $\pi$ and by $n_1, \ldots, n_m$ the cardinalities of the blocks. Using Lemma 2.5 we are ready to obtain asymptotics of the mean of $S_N(f).$

\begin{lemma}
\begin{align}
\label{mathexp2}
\E S_N(f)=\mathcal{M}(f)\*N +O(1),
\end{align}
where
\begin{align}
\label{c2}
& \mathcal{M}(f)=  \frac{1}{(2\pi)^{n/2}}\*\sum_{m=1}^n\* \frac{1}{m!}\*\sum_{\substack{(n_1, \ldots, n_m):
\\ n_1+\ldots +n_m=n+1, \ n_1, \ldots, n_m\geq 1}}  \* \frac{(n+1)!}{n_1!\cdots n_m!}\times\\ 
& 
\int_{L_{\pi}} \hat{f}(t_1, \ldots, t_n)\*
  \prod_{j=1}^m c_{n_j}(t_{M_{j-1}+1}, \ldots, t_{M_j-1}, -t_{M_{j-1}+1}-\ldots - t_{M_j-1}) \*d\lambda,\nonumber
 \end{align}
$ t_{n+1}=-\sum_{i=1}^n t_i, \ \ M_j=n_1+\ldots+n_j, \ 1\leq j\leq m,\  M_0=0, \ $ 
$L(n_1, \ldots, n_m)\subset \R^{n+1}$ is defined as
\begin{align}
\label{lblb}
L_{n_1, \ldots, n_m}:=\{t=(t_1, \ldots, t_{n+1})\in\R^{n+1} : \ \sum_{i=1}^{n+1} t_i=0; \ \ \sum_{M_{j-1}<i\leq M_j} t_i=0, \ \ 1\leq j\leq m\},
\end{align}
$\ \lambda \ $  is the Lebegsue measure on $L_{\pi},$  i.e.
\begin{align}
\label{lebmeas}
\lambda=\prod_{j: n_j>1} dt_{M_{j-1}+1}\cdots dt_{M_j-1},
\end{align}
and rescaled joint cumulant functions $c_{n_j}$'s are defined in (\ref{j1}-\ref{j2}).
 \end{lemma} 
 \begin{remark}
 For singleton blocks $B: |B|=n_j=1, \ $ we write $c_{n_j}(t_{M_{j-1}+1})$ in (\ref{c2}) and use $c_{n_j}(t_{M_{j-1}+1})=c_{n_j}(0)=1$ on $L_{\pi}.$
 \end{remark}
 \vskip .4cm

 Next, we study $\var S_N(f).$ It follows from (\ref{f2}) that
\begin{align}
\label{v1} 
& \var(S_N(f))= \frac{1}{(2\pi)^{n}\*N^{2n}}\* \sum_{k^{(1)}\in \mathbb{Z}^n} \* \sum_{k^{(2)}\in \mathbb{Z}^n} \* \hat{f}(k^{(1)}/N)\* \hat{f}(k^{(2)}/N) \times\\
& 
 \left(\E \left[
\prod_{j=1}^{n+1} T_{N, k_j}\*\prod_{j=1}^{n+1} T_{N, k_{n+1+j}}\right] - \E\left[
\prod_{j=1}^{n+1} T_{N, k_j}\right]\* \E\left[
\prod_{j=1}^{n+1} T_{N, k_{n+1+j}}\right]  \right),\nonumber
\end{align}  
where $k^{(1)}=(k_1, \ldots, k_{n}), \ \ k^{(2)}=(k_{n+2}, \ldots, k_{2n+1})\in\Z^{n}, \ \ k_{n+1}:=-\sum_{i=1}^n k_i,$ \\ $k_{2n+2}:=-\sum_{i=1}^n k_{n+1+i},$
and $T_{N,s}$ is defined in (\ref{deftrace}).

\vskip .2cm
To study the covariance of the products of traces in (\ref{v1}), we rewrite it using (\ref{cumtomom}) as
\begin{align}
\label{v2}
\E \left[
\prod_{j=1}^{n+1} T_{N, k_j}\*\prod_{j=1}^{n+1} T_{N, k_{n+1+j}}\right] - \E\left[
\prod_{j=1}^{n+1} T_{N, k_j}\right]\* \E\left[
\prod_{j=1}^{n+1} T_{N, k_{n+1+j}}\right]=
\sum_{\pi}^* \prod_{B\in \pi} \kappa_{|B|}^{(N)}(k_i: i \in B),
\end{align}
where the sum is over all partitions $\pi$ of $\{1,2,\ldots, 2n+2\},$ satisfying the condition (\ref{condition}) below, $B$ runs through the list of all blocks of the partition $\pi, \ $ and $|B|$ is the cardinality of a block $B.\ $  The condition on $\pi$ is that it is not a union of partitions of $\{1, \dots, n+1 \}$ and $\{n+2, \ldots, 2n+2\},$ in other words:
\begin{align}
\label{condition}
 & {\text The \ set \ }\{1,2,\ldots, n+1\}  {\text \ is \ not \ a \ union \ of \ some \ blocks \ 
 of \ } \pi. 
\end{align}
Similarly to the $\E S_N(f)$ computations, define
\begin{align}
\label{lb2}
L_{\pi}:=\{t=(t_1, \ldots, t_{2n+2})\in\R^{2n+2} : \sum_{i=1}^{n+1} t_i=0, \  \sum_{i=1}^{n+1} t_{n+1+i}=0; \ \ \sum_{i\in B} t_i=0 \ \forall B\in \pi \}.
\end{align}
Using (\ref{v2}) we obtain
\begin{align}
\label{v10} 
\var(S_N(f)) & = \frac{1}{(2\pi)^{n}\*N^{2n}}\* \sum_{k^{(1)}\in \mathbb{Z}^n} \* \sum_{k^{(2)}\in \mathbb{Z}^n} \* \hat{f}(k^{(1)}/N)\* \hat{f}(k^{(2)}/N) \*
\sum_{\pi}^* \prod_{B\in \pi} \kappa_{|B|}^{(N)}(k_i: i \in B)\\
& = \frac{1}{(2\pi)^{n}}\* \sum_{\pi}^* N^{|\pi|-2\*n} \*\sum_{k^{(1)}\in \mathbb{Z}^n} \* \sum_{k^{(2)}\in \mathbb{Z}^n} \*\hat{f}(k^{(1)}/N)\* \hat{f}(k^{(2)}/N) \*\prod_{B\in \pi} c_{|B|}(k_i/N: i \in B). \nonumber
\end{align} 

\begin{remark}
As before, the sum $\sum_{\pi}^*$ in (\ref{v10}) is over all partitions $\pi$ of $\{1,2,\ldots, 2n+2\}$ satisfying the condition (\ref{condition}). The shorthand notation 
 $c_{|B|}(t_i: i\in B)$  means that the arguments of $c_{|B|}$ are the $t_i$'s corresponding to $i\in B.$
 Since the joint cumulant functions are symmetric the order of the variables is not important.
\end{remark}

It follows from Lemma 2.5, (ii)-(iii), that for any fixed $\pi$ the summation in (\ref{v10}) is over
$k=(k_1, \ldots, k_{2n+2}) \in L_{\pi} \cap \frac{1}{N}\*\Z^{2n+2},$
since $k$ satisfies the following system of linear equations of rank $|\pi|+1:$
\begin{align}
\label{lineq2}
\begin{cases}
& \sum_{i=1}^{n+1} k_i=0, \\
& \sum_{i=1}^{n+1+i} k_i=0, \\
& \sum_{i\in B} k_i=0, \ \ \forall B\in \pi.
\end{cases}
\end{align}
The first linear equation $\sum_{i=1}^{n+1} k_i=0$ follows from the remaining $|\pi|+1$ independent linear equations in (\ref{lineq2}). Therefore
\begin{align}
\label{dimll}
\dim L_{\pi}=2n+2-(|\pi|+1)=2n+1-|\pi|.
\end{align}
Proceeding as in the case of the mathematical expectation above we arrive at
\begin{lemma}
\begin{align}
\label{v3}
\var(S_N(f))=\sigma^2(f)\*N +O(1),
\end{align}
where
\begin{align}
 \label{v4}
 \sigma^2(f)= \frac{1}{(2\pi)^{n}}\*\sum_{\pi}^* \prod_{B\in \pi} \*\int_{L_{\pi}} \ 
 f(t^{(1)})\*f(t^{(2)})\*\prod_{B\in\pi} 
 c_{|B|}(t_i: i\in B) \ \* d\lambda,
    \end{align}
where the sum in (\ref{v4}) is over all partitions $\pi$ of $\{1,2,\ldots, 2n+2\},$ satisfying the condition (\ref{condition}), $B$ runs through the list of all blocks of the partition $\pi, \ \ |B|$ is the cardinality of a block $B,\  \ t^{(1)}=(t_1, \ldots, t_n), \ t^{(2)}=(t_{n+2}, \ldots, t_{2n+1}),$
and $\lambda$ is the Lebesgue measure on $L_{\pi},$ i.e. it is the product of $dt_i$'s taken over all independent variables $t_i$ from the system of linear equations
\begin{align}
\label{lineq22}
\begin{cases}
& \sum_{i=1}^{n+1} t_i=0, \\
& \sum_{i=1}^{n+1+i} t_i=0, \\
& \sum_{i\in B} t_i=0, \ \ \forall B\in \pi.
\end{cases}
\end{align}
\end{lemma}

\section{Proof of the Main Result}
The section is devoted to the proof of Theorem 1.3. The proof uses the method of moments and is combinatorial in nature.

\begin{proof}
Fix a positive integer $m>2.$ We have
\begin{align}
\label{momentm} 
& \E (S_N(f)-\E S_N(f))^m= \frac{1}{(2\pi)^{\frac{n\*m}{2}}\*N^{m\*n}}\* \sum_{k^{(1)}\in \mathbb{Z}^n} \ldots \sum_{k^{(m)}\in \mathbb{Z}^n} \* \hat{f}(k^{(1)}/N)\cdots \hat{f}(k^{(m)}/N) \times\\
& 
 \E \left[\prod_{i=0}^{m-1} \left(
\prod_{j=1}^{n+1} T_{N, k_{(n+1)\*i+j}}-
\E \prod_{j=1}^{n+1}  T_{N, k_{(n+1)i+j}}\right) \right],\nonumber
\end{align}  
$k^{(1)}=(k_1, \ldots, k_{n}), \ldots , k^{(m)}=(k_{(m-1)\*(n+1)+1}, \ldots, k_{m\*(n+1)-1})\in\Z^{n}, $\\ $k_{n+1}:=-\sum_{j=1}^n k_j, \ \ k_{2(n+1)}=-\sum_{j=1}^{n} k_{n+1+j},\ldots, k_{m(n+1)}:=- \sum_{j=1}^n k_{(m-1)\*(n+1)+j}.$

The mathematical expectation on the r.h.s. of (\ref{momentm}) can be written in terms of joint cumulants 
using the following lemma.
\begin{lemma}
For centered random variables
$ X_1, \ldots, X_{m\*l}$ with finite moments,
\begin{align}
\label{cenmoments1}
\E \left[\prod_{i=0}^{m-1} \left(
\prod_{j=1}^{l} X_{i\*l+j}-
\E \prod_{j=1}^{l}  X_{i\*l+j}\right) \right]
= \sum^*_{\pi} \prod_{B\in \pi} \kappa(X_i: i \in B),
\end{align}
where the sum on the r.h.s. of (\ref{cenmoments1}) is over all partitions $\pi$ of $\{1, \ldots, m\*l\}$ that do not contain a partition of $\{ i\*l +1, \ldots,  (i+1)\*l\}$ for any $0\leq i\leq m-1,$ i.e. none of the sets $\{ i\*l +1, \ldots,  (i+1)\*l\}$  can be represented as a union of some blocks of $\pi.$
\end{lemma}

\begin{proof}

It follows from the formula expressing moments in  terms of cumulants (see e.g. (\ref{cumtomom})) that the r.h.s. of (\ref{cenmoments1}) is equal to a linear combination of $\prod_{B\in \pi} \kappa(X_i: i \in B),$  where
$\pi$ runs over the list of partitions of $\{1,2,\ldots, m\*l\}.$  Thus our goal is to show that a coefficient in the linear combination is either $1$ or $0$ depending on whether $\pi$ satisfies the condition in Lemma 4.1 or not.

If  any  sub-collection of $\pi$ is not a partition of $ \{ i\*l +1, \ldots,  (i+1)\*l\}$ for any  $\ i=1, \ldots,m,$ then the coefficient in front of the product 
$ \prod_{B\in \pi} \kappa(X_i: i \in B)$ in the linear 
combination is $1$ since the only contribution comes from $\E \left [\prod_{i=0}^{m-1} 
\prod_{j=1}^{l} X_{i\*l+j}\right].$

Finally, suppose $s, \ 1\leq s\leq m,$ of the sets $ \{ i\*l +1, \ldots,  (i+1)\*l\}, \ 1\leq i \leq m,$ 
can be represented as unions of some blocks of $\pi.$
Then the coefficient in front of  $\prod_{B\in \pi} \kappa(X_i: i \in B)$ is equal to  
\begin{align}
\sum_{k=0}^s (-1)^k \frac{s!}{k!\*(s-k)!}=0.
\end{align}
\end{proof}

Following Lemma 4.1  we can rewrite the $m$-th centered moment as
\begin{align}
\label{momentmm} 
&
\E (S_N(f)-\E S_N(f))^m= \\
& \frac{1}{(2\pi)^{\frac{n\*m}{2}}\*N^{m\*n}}\* \sum_{k^{(1)}\in \mathbb{Z}^n} \ldots \sum_{k^{(m)}\in \mathbb{Z}^n} \* \hat{f}(k^{(1)}/N)\cdots \hat{f}(k^{(m)}/N) \*
\sum_{\pi}^* \prod_{B\in \pi} \kappa_{|B|}^{(N)}(k_i: i \in B),\nonumber
\end{align}
where, as before, $k^{(1)}=(k_1, \ldots, k_{n}),  \ldots , k^{(m)}=(k_{(m-1)\*(n+1)+1}, \ldots, k_{m\*(n+1)-1})\in\Z^{n},$\\ $ k_{n+1}=-\sum_{j=1}^n k_j, \ldots, k_{m(n+1)}=- \sum_{j=1}^n k_{(m-1)\*(n+1)+j}.$

As in Section 3, we consider a linear subspace
\begin{align}
\label{lb3}
L_{\pi}:=\{t\in\R^{m\*(n+1)} : \sum_{j=1}^{n+1} t_{i(n+1)+j}=0, \  \forall
\ 0\leq i\leq m-1; \ \ \sum_{j\in B} t_j=0 \ \forall B\in \pi \}.
\end{align}

It follows from Lemma 2.5, (ii)-(iii), that for any fixed $\pi$ the summation in (\ref{momentmm}) is over
$k=(k_1, \ldots, k_{m\*(n+1)}) \in L_{\pi} \cap \frac{1}{N}\*\Z^{m\*(n+1)},$
since $k$ satisfies the following system of linear equations:
\begin{align}
\label{lineq3}
\begin{cases}
& \sum_{j=1}^{n+1} k_{i\*(n+1)+j}=0, \ \forall \ 0\leq i\leq m-1, \\
& \sum_{i\in B} k_i=0, \ \ \forall B\in \pi.
\end{cases}
\end{align}

Using (\ref{scalcum}), we rewrite (\ref{momentmm}) as
\begin{align}
\label{momentmmm} 
&
\E (S_N(f)-\E S_N(f))^m= \\
& \frac{1}{(2\pi)^{\frac{n\*m}{2}}\*N^{m\*n}}\* \sum_{k\in L^{\pi}\cap \frac{1}{N}\*\mathbb{Z}^{m\*(n+1)}} \* \hat{f}(k^{(1)}/N)\cdots \hat{f}(k^{(m)}/N) \*
\sum_{\pi}^* \ N^{|\pi|}\*\prod_{B\in \pi} c_{|B|}(k_i/N: i \in B)=\nonumber \\
& \frac{1}{(2\pi)^{\frac{n\*m}{2}}}\sum_{\pi}^* \* N^{|\pi|-m\*n}\* \sum_{k\in L^{\pi}\cap \frac{1}{N}\*\mathbb{Z}^{m\*(n+1)}} \* \hat{f}(k^{(1)}/N)\cdots \hat{f}(k^{(m)}/N) \*\prod_{B\in \pi} c_{|B|}(k_i/N: i \in B).\nonumber
\end{align}

The crucial question in the power counting analysis of the r.h.s. of (\ref{momentmmm})  is the dimension of the vector subspace $L_{\pi}\subset \R^{m\*(n+1)},$ or, equivalently,  the rank of the system of linear equations  (\ref{lineq3}).

Denote
\begin{align}
[1]:=\{1,\ldots, n+1\}, \ [2]:=\{n+2, \ldots, 2n+2\}, \ldots, [m]=\{(m-1)\*(n+1)+1, \ldots, m\*(n+1)\}.
\nonumber
\end{align}
\begin{definition}
For a given partition $\pi$ of $\{1, \ldots, m\*n\}$ an equivalence relation $\sim_{\pi}$ is defined 
on the set $\{[1], [2], \ldots, [m]\}$ in the following way:
\begin{align}
\label{eqrel}
[i]\sim_{\pi} [j]    
\end{align}
if and only if there is a block $B$ in the partition $\pi$ such that $B\cap[i]\neq \emptyset$ and
$B\cap[j]\neq \emptyset.$

\end{definition}
\begin{remark}
It follows from Lemma 4.1 that the cardinality of each equivalence class of the equivalence relation $\sim_{\pi}$ is at least 2.
\end{remark}
\begin{definition}
We call a partition $\pi$
optimal if the cardinality of every equivalence class of the equivalence relation $\sim_{\pi}$ is 2. 
If $\pi$ is not optimal, it is called sub-optimal.
\end{definition}
Clearly, optimal partitions exist if and only if $m$ is even.
Next lemma is the main ingredient of the proof of Theorem 1.3.

\begin{lemma}
(i)
Let $m$ be an even positive integer and $\pi$ be an optimal partition of $\{1, \ldots, m\*n\}$. Then a linear subspace $L_{\pi} \subset \R^{m\*(n+1)}$ defined in (\ref{lb3}) satisfies
\begin{align}
\label{dimdim}
\dim L_{\pi}= m\*n +m/2 -|\pi|.
\end{align}

(ii) Let $m>1$ be a positive integer and $\pi$ be an sub-optimal partition of $\{1, \ldots, m\*n\}$.
Then
\begin{align}
\label{dimdimdim}
\dim L_{\pi}<m\*n +m/2 -|\pi|.
\end{align}
\end{lemma}
\begin{proof}

(i)  If $\pi$ is optimal it can be viewed as a union of $\frac{m}{2}$ partitions $\pi_{i,j}$ of $[i]\cup[j], \  [i]\sim_{\pi}[j].$ Each partition $\pi_{i,j}$ corresponds to a vector subspace $L_{\pi_{i,j}}\subset \R^{2n+2}$ of dimension 
\[
\dim L_{\pi_{i,j}}=2\*n+1-|\pi_{i,j}|
\]
(see (\ref{lb2}) and (\ref{dimll}).) Since $L_{\pi}$ is the Cartesian product of 
$L_{\pi_{i,j}}$'s  we have
\begin{align}
\dim  L_{\pi}   =\sum \dim L_{\pi_{i,j}},
\end{align}
and (\ref{dimdim}) immediately follows.

(ii) Now let us assume that $\pi$ is a sub-optimal partition. Recall that the subspace $L_{\pi}$ 
is determined by the following system of linear equations
\begin{align}
\label{lineq222}
\begin{cases}
& \sum_{j=1}^{n+1} t_{i\*(n+1)+j}=0,  \ \ \forall \ 0\leq i\leq m-1,\\
& \sum_{j\in B} t_j=0, \ \ \forall B\in \pi.
\end{cases}
\end{align}
To prove (\ref{dimdimdim}) we have to show that the rank of this system is bigger than $|\pi|+\frac{m}{2}.\ $ As before, $L_{\pi}$ can be viewed as the Cartesian product of the subspaces corresponding to the equivalence classes. Using the additivity of dimension, we can assume without a loss of generality that $\pi$ in (\ref{lineq222}) has only one equivalence class. We claim that in this case the rank of (\ref{lineq222}) is $|\pi|+m-1>|\pi|+\frac{m}{2},$ since $m>2.$  To show this, consider $|\pi|+m-1$ vectors in $\R^{m\*(n+1)},\ $ namely $\{ \chi_B, \ B\in \pi\} \cup \{ \chi_{[i]}, \ 0\leq i<m-1\}:$
\begin{align}
 & \chi_B(j)=\begin{cases}
 & 1 \ \ j\in B, \\
 & 0 \ \ j \notin B,
 \end{cases}    \\
 & \chi_{[i]}(j)=\begin{cases}
 & 1 \ \ j\in [i]=\{i\*(n+1) +1, \ldots, (i+1)\*(n+1)\}, \\
 & 0 \ \ j \notin [i].
 \end{cases}    
\end{align} 
Part (ii) of Lemma 4.5 follows from
\begin{lemma}
Let $\pi$ be a sub-optimal partition of $\{1, \ldots, m\*n\},  \ \ m\geq 1,  \ $ such that the equivalence relation $\sim_{\pi}$ 
on $\{[1], \ldots, [m]\}$ has only one equivalence class. Then
the vectors 
\begin{align}
\label{vectors}
\{ \chi_B, \ B\in \pi\} \cup \{ \chi_{[i]}, \ 0\leq i<m-1\}
\end{align}
are linearly independent.
\end{lemma}
\begin{proof}
The vectors $\{ \chi_B, \ B\in \pi\}$ are linearly independent  since their supports are disjoint. We show by induction that adding
$\chi_{[i]}$ one by one preserves linear independence.  Suppose that 
$\{ \chi_B, \ B\in \pi\} \cup \{ \chi_{[i]}, \ 0\leq i<k-1\}, \ k<m, \ $ are linearly independent  and $\chi_{[k]}$ can be written as a linear combination of $\{ \chi_B, \ B\in \pi\} \cup \{ \chi_{[i]}, \ 0\leq i<k-1\}.$ Then a non-trivial linear combination of $\{ \chi_{[i]}, \ 0\leq i<k\} \ $ can be written as a linear combination of some of the vectors $\chi_B.$ This implies that $[1]\cup \ldots \cup[k]$
contain an equivalence class of the equivalence relation $\sim_{\pi}, \ $ which is a contradiction.
Lemma 4.6 is proven.
\end{proof}
Since $L_{\pi}$ is the orthogonal complement of the linear span of the vectors (\ref{vectors})
this implies part (ii) of Lemma 4.5.
\end{proof}
Now we are ready to finish the proof of Theorem 2.5. Let $m=2\*k$ be even. 
Denote a subsum in (\ref{momentmmm}) corresponding to a partition $\pi$ by $\Sigma_{\pi}.$
It follows from Lemma 4.5 that only the optimal partitions $\pi$ give leading contribution of order $N^{m/2}$ to the r.h.s. of (\ref{momentmmm}). For each sub-optimal $\pi$ we have
$ \ \Sigma_{\pi}=  O(N^{\frac{m-1}{2}}). \ $

There are exactly 
$(2\*k-1)!!$ ways to split the set $\{[1], \ldots, [2\*k] \}$ into pairs. 
By repeating the variance computations, we note that the sum of  $\Sigma_{\pi}$'s  over all $\pi$ corresponding to any particular splitting of $\{[1], \ldots, [2\*k] \}$ into pairs gives 
$ \ \sigma^{m}(f)\*N^{\frac{m}{2}} \*(1+o(1)). \ $

We conclude that the $2\*k$-th moment of
$\frac{S_N(f)-\E S_N(f)}{\sqrt{\var S_N(f)}}$ converges to $(2k-1)!!$ in the limit $N\to \infty.$
The case of odd $m$ is treated similarly. Specifically, one obtains that the odd moments of
$\frac{S_N(f)-\E S_N(f)}{\sqrt{\var S_N(f)}}$ converges to $0$ in the limit $N\to \infty.$

This finishes the proof of Theorem 2.5.
\end{proof}

\end{document}